\documentclass{birkjour}
\usepackage{amsfonts,amssymb}
\usepackage{xcolor}
 \newtheorem{thm}{Theorem}[section]
 \newtheorem{cor}{Corollary}[thm]
 
\newtheorem*{thrm}{Theorem}

\newtheorem*{lemm}{Lemma}
\newtheorem{conj}{Conjecture}
 
 \theoremstyle{definition}
 \newtheorem{defn}[thm]{Definition}
 \theoremstyle{remark}
\newtheorem{exmp}{Example}
\newtheorem{rem}{Remark}
 \numberwithin{equation}{section}

\begin{document}

%-------------------------------------------------------------------------
% editorial commands: to be inserted by the editorial office
%
%\firstpage{1} \volume{228} \Copyrightyear{2004} \DOI{003-0001}
%\seriesextraline{This is the Concrete Title of this Book\br H.E. R and S.T.C. W, Eds.}
%\firstpage{1}

\title[Denjoy-Wolff like set for rational semigroups]
 {Denjoy-Wolff like set for rational semigroups}

%----------Author 1
\author[S. Chatterjee]{Subham Chatterjee}

\address{%
Department of Mathematics,\\ Sidho-Kanho-Birsha University,\\
Purulia, 723104, West Bengal, India.}

\email{subham.acad.res@gmail.com}

%----------Author 2
\author[G. Chakraborty]{Gorachand Chakraborty}
\address{Department of Mathematics,\\ Sidho-Kanho-Birsha University,\\
Purulia, 723104, West Bengal, India.}
\email{Corresponding author: gorachand.chakraborty@skbu.ac.in}

%----------Author 3
\author[T. K. Chakra]{Tarun Kumar Chakra}
\address{Department of Mathematics,\\ Indian Institute of Technology
Madras,\\ Chennai, 600036, Tamil Nadu, India.}
\email{tarunchakra1987@gmail.com}


\date{ }
%\dedicatory{ }
%%% ----------------------------------------------------------------------

\begin{abstract}
In this paper, we introduce the concept of Denjoy-Wolff set in rational semigroups. We show that for finitely generated Abelian rational semigroups, the Denjoy-Wolff like set is countable. Some results concerning the Denjoy-Wolff like set and the Julia set are also discussed. Then we consider a special class of rational semigroups and discuss various properties of Denjoy-Wolff like set for this class. We use the concept of Denjoy-Wolff like set to classify the class into 3 sub-classes. We also show that for any semigroup in this class, the semigroup can be partitioned into $k$ partitions where $k$ is the cardinality of the Denjoy-Wolff like set. 
\end{abstract}

%%% ----------------------------------------------------------------------
\maketitle
%%% ----------------------------------------------------------------------
%\tableofcontents
\section{Introduction}
Let $f:\widehat{\mathbb{C}}\rightarrow \widehat{\mathbb{C}}$ be a rational function. By the dynamics of the function $f$ we mean the behaviour of the sequence $\{f^n\}$ where $f^n=f^{n-1}\circ f$, $n\in\mathbb{N}$ and $f^0$ is the identity function. The Fatou set of a function $f$ is the set $F(f)=\{z\in\widehat{\mathbb{C}}: \{f^n\}\text{ is well-defined and normal in a neighbourhood of } z\}$. The complement of the Fatou set is called the Julia set, denoted by $J(f)$. The Fatou and Julia set of a rational function divides the Riemann sphere into two disjoint completely invariant subsets. For a detailed discussion about the dynamics of rational functions, one can read \cite{beardon2000iteration}. While investigating the dynamics of rational functions, the topic of rational semigroups was introduced in \cite{hinkkanen1996dynamics}. A rational semigroup is a collection of rational maps of degree greater than 1 which is closed under the binary operation of composition. To define rational semigroups symbolically, let $F=\{f_1,f_2,\dots\}$ be a family of rational functions of degree at least $2$. Let $G$ be the collection of all functions $f$ such that $f=f_{i_1}\circ f_{i_2}\circ\dots\circ f_{i_n}$ for some $i_k\in\mathbb{N}$. Then $G$ forms a semigroup under composition of functions and is called a rational semigroup. The family $F$ is called the generating family of $G$ and the semigroup is denoted as $G=<f_1,f_2,\dots>=<F>$. When the generating family $F$ is finite then we say $G$ is a finitely generated rational semigroup and $G$ is called infinitely generated if the generating family is infinite. If $F=\{f_1,f_2,\dots\}$ is a family of permutable rational functions that is $f_i\circ f_j=f_j\circ f_i$ for all $i,j\in\mathbb{N}$ and $G=<F>$ then we say $G$ is an Abelian rational semigroup. By the Fatou set of a rational semigroup $G$, we mean the largest open subset of $\widehat{\mathbb{C}}$ where the family $G$ is well-defined and normal. The Fatou set of a rational semigroup $G$ is denoted by $F(G)$. The complement of the Fatou set of $G$ is called the Julia set of $G$ and is denoted by $J(G)$. The idea of rational semigroups was extended to transcendental entire functions in \cite{poon1998fatou}. It is known from \cite{hinkkanen1996dynamics,poon1998fatou} that $J(G)=\overline{\cup_{g\in G} J(g)}$ for any rational or transcendental semigroup $G$ where $\overline{A}$ denotes the closure of $A$. Recently the concept of omitted value for transcendental semigroups is introduced in \cite{chatterjee2023onomitted}. In this paper, by $A^\circ$ and $A^c$ we mean the set of interior points of $A$ and complement of $A$ respectively. The cardinality of a set $A$ is denoted as $card(A)$ or $\lvert A\rvert$. Throughout this paper, we denote the open unit disk and the unit circle centered at the origin by $\mathbb{D}$ and $\partial \mathbb{D}$ respectively. A set $B$ is said to be forward invariant for a rational or transcendental semigroup $G$ if $f(B)\subset B$ for all $f\in G$. Similarly, a set $B$ is said to be backward invariant for a rational or transcendental semigroup $G$ if $B\subset f(B)$ for all $f\in G$. It is known that the Fatou set is forward invariant whereas the Julia set is backward invariant for rational semigroups. A set that is both forward and backward invariant is called a completely invariant set. Some results on completely invariant Julia sets in polynomial semigroups can be found in \cite{stankewitz1999completely}. \\
M\"{o}bius maps are the rational maps of the form $T(z)=\frac{az+b}{cz+d},\text{ where }a,b,c,d\in\mathbb{C},ad-bc\neq 0$ and $z\in\widehat{\mathbb{C}}$. It is known that non-identity M\"{o}bius maps can be classified into $3$ types. A M\"{o}bius map $T$ is called parabolic if $T(z)$ is conjugate to $f(z)=z+\beta$ $\beta\neq 0$ and has only one fixed point. M\"{o}bius map $T$ is called elliptic if $T(z)$ is conjugate to $f(z)=\alpha z$ with $\rvert \alpha\lvert=1$ and has two fixed points. A M\"{o}bius map is called loxodromic in the other cases. M\"{o}bius map $T$ is called hyperbolic if $T(z)$ is conjugate to $f(z)=\alpha z$ with $\alpha\in (0,1)\cup (1,\infty)$. One of the most well-known results in complex analysis and complex dynamics is the Denjoy-Wolff theorem stated below. It shows the existence of a unique point in $\overline{\mathbb{D}}$ to which all the orbits of points in the unit disk converge for a particular type of holomorphic self-maps of the unit disk. 
\begin{thrm}[Denjoy-Wolff theorem]
Let $f:\mathbb{D}\rightarrow\mathbb{D}$ be analytic and assume that $f$ is neither an elliptic M\"{o}bius map nor the identity. Then there exists a unique point $z_0\in\overline{\mathbb{D}}$ such that $f^n(z)\rightarrow z_0$ as $n\rightarrow \infty$ uniformly on compact subsets of $\mathbb{D}$.
\end{thrm}
Denjoy-Wolff theorem is very useful in describing the behaviour of iterations of holomorphic self-maps of the unit disk. For more details on this, one can read \cite{benini2024boundary,bracci2020continuous}. For non-rotational holomorphic self-maps of $\mathbb{D}$, this convergence can even be extended to the boundary almost everywhere under certain assumptions \cite{poggi2010pointwise}. Let $f:\mathbb{D}\rightarrow\mathbb{D}$ be a conformal map. 
The set of all conformal maps $f:\mathbb{D}\rightarrow\mathbb{D}$ is the group of M\"{o}bius maps given by $f(z)=e^{i\alpha}\frac{z-a}{1-\overline{a}z}, \alpha\in\mathbb{R}, \lvert a\rvert<1$. \\
Our goal in this paper is to generalize the aspects of the Denjoy-Wolff theorem in the study of rational semigroups. For this purpose, we define a set consisting of all Denjoy-Wolff points for the rational semigroup. 
\begin{defn}
Let $G=<f_1,f_2,\dots>$ be a rational semigroup. We define the Denjoy-Wolff like set as $DW(G)=\{z^{\prime}\in\overline{\mathbb{D}}:\text{ there exists } f\in G \text{ such that } f^n(z)\rightarrow z^{\prime}\text{ as } n\rightarrow\infty \text{ for all } z\in\mathbb{D}\}$
\end{defn}
In Section \ref{s2}, we provide examples of various rational semigroups and discuss the cardinality of Denjoy-Wolff like set for them. We show that for a finitely generated Abelian rational semigroup, the Denjoy-Wolff like set is at most countably infinite. We also propose a conjecture that the Denjoy-Wolff like set for finitely generated Abelian rational semigroups is either empty or singleton. In Section \ref{s3}, we discuss some properties of Denjoy-Wolff like set for various rational semigroups. In Section \ref{section4}, a special class of rational semigroups is considered. We discuss the structure of the class based on the Denjoy-Wolff like set and  classify it into 3 categories based on the behaviour of their corresponding Denjoy-Wolff like set. Some results about these subclasses are also discussed.

\section{Denjoy-Wolff like set for rational semigroups}\label{s2}
Let $G$ be a rational semigroup. We first provide two examples showing that $DW(G)$ can be empty or finite.
\begin{exmp}\label{example1}
Let $G=<f_1,f_2>$ where $f_1(z)=z^2$ and $f_2(z)=z^3$. Let $f\in G$. So $f(z)=z^{2^n 3^m}$ for some $n,m\in\mathbb{N}$. Then $f^k(z)\rightarrow 0$ as $k\rightarrow \infty$ for any $z\in\mathbb{D}$. So $DW(G)=\{0\}$ as $f$ is chosen arbitrarily.
\end{exmp}
\begin{exmp}\label{example2}
Let $G=<f>$ where $f(z)=z^2+1$. Consider $z=\frac{1}{2}$. We see that $f^n(\frac{1}{2})$ does not converge to any point in $\overline{\mathbb{D}}$. So, $DW(G)$ is empty.  
\end{exmp}
Next, we provide an example of a rational semigroup whose Denjoy-Wolff like set has an infinite number of elements. For this, we first take a family of functions which are simply the composition of a conformal self-map of $\mathbb{D}$ with a polynomial.
\begin{rem}\label{le1}
    If $f(z)=\frac{z^k+a}{1+\overline{a}z^k}$, $\lvert a\rvert<1$, $k\in\mathbb{N}$ then $f(\mathbb{D})\subset \mathbb{D}$. This is obvious as $f=g\circ h$ where $g(z)=\frac{z+a}{1+\overline{a}z}$ and $h(z)=z^k$. Clearly, both $g,h$ map $\mathbb{D}$ to $\mathbb{D}$ and thus $f$ also does the same. By Denjoy-Wolff theorem there exists a $z_0\in\overline{\mathbb{D}}$ such that $f^n(z)\to z_0$ for all $z\in\mathbb{D}$ when $k\geq 2$. Also, notice that the poles of $f$ lie outside the unit disk.
\end{rem}
\begin{thm}\label{le2}
    Let $f(z)=\frac{z^k+a}{1+\overline{a}z^k}$, $0<\lvert a\rvert<1, k\in\mathbb{N}\setminus \{1\}$. Then,\\ 
    i) $\frac{a}{\lvert a\rvert}$ is a fixed point of $f$ if and only if $\frac{a}{\lvert a\rvert}$ is $(k-1)^{th}$ root of unity.\\
    ii) If $\frac{a}{\lvert a\rvert}$ is $(k-1)^{th}$ root of unity, then $\frac{a}{\lvert a\rvert}$ is a parabolic fixed point if and only if $\lvert a\rvert=\frac{k-1}{k+1}$.\\
    iii) If $\frac{a}{\lvert a\rvert}$ is $(k-1)^{th}$ root of unity and $\frac{a}{\lvert a\rvert}$ is a parabolic fixed point, then $f^n(z)\to \frac{a}{\lvert a\rvert}$ for all $z\in\mathbb{D}$ as $n\to\infty$.
\end{thm}
\begin{proof}
    \textit{i)} First let us assume $\frac{a}{\lvert a\rvert}$ is a fixed point of $f$. Now, $f(z)=z$ implies $z^k=\frac{z-a}{1-\overline{a}z}$. As $\frac{a}{\lvert a\rvert}$ is a fixed point of $f$, so $\frac{a^k}{\lvert a\rvert^k}=\frac{\frac{a}{\lvert a\rvert}-a}{1-\overline{a}\frac{a}{\lvert a\rvert}}=\frac{a}{\lvert a\rvert}$, as $1-\lvert a\rvert>0$. So $\frac{a}{\lvert a\rvert}\Bigl(\bigl(\frac{a}{\lvert a\rvert}\bigl)^{k-1}-1\Bigl)=0$. As $\frac{a}{\lvert a\rvert}\neq 0$, $\Bigl(\frac{a}{\lvert a\rvert}\Bigl)^{k-1}=1$. Hence, $\frac{a}{\lvert a\rvert}$ is $(k-1)^{th}$ root of unity. Conversely, if $\frac{a}{\lvert a\rvert}$ is $(k-1)^{th}$ root of unity then by minor calculation we have $f\Bigl(\frac{a}{\lvert a\rvert}\Bigl)=\frac{a}{\lvert a\rvert}$. \\
    \textit{ii)} We have, $\frac{a}{\lvert a\rvert}$ is a fixed point of $f$. Now, $f^{\prime}(z)=\frac{kz^{k-1}(1-\lvert a\rvert^2)}{(1+\overline{a}z^k)^2}$. So, $f^{\prime}\Bigl(\frac{a}{\lvert a\rvert}\Bigl)=k\frac{1-\lvert a\rvert}{1+\lvert a\rvert}$ as $\frac{a}{\lvert a\rvert}$ is $(k-1)^{th}$ root of unity. Now, $\frac{a}{\lvert a\rvert}$ is a parabolic fixed point if and only if $f^{\prime}\Bigl(\frac{a}{\lvert a\rvert}\Bigl)=1$. Thus $\frac{a}{\lvert a\rvert}$ is a parabolic fixed point if and only if $\lvert a\rvert=\frac{k-1}{k+1}$.\\
    \textit{iii)} As $\frac{a}{\lvert a\rvert}$ is a parabolic fixed point of $f$, so $\frac{a}{\lvert a\rvert}\in \partial \mathbb{D}\cap J(f)$. Also, as $f$ satisfies conditions of the Denjoy-Wolff theorem so $f^n(z)\to z_0\in\overline{\mathbb{D}}$. Thus $\mathbb{D}\subset F(f)$. Let $\mathbb{D}\subset U$, where $U$ is a Fatou component. As $\frac{a}{\lvert a\rvert}\in\partial \mathbb{D}\cap J(f)$, $\partial\mathbb{D}$ is not subset of $U$. So, $\frac{a}{\lvert a\rvert}\in\partial \mathbb{D}\cap \partial U$. So, $U$ is a Fatou component that contains a parabolic fixed point in its boundary. So $U$ is a parabolic basin such that $f^n(z)\to \frac{a}{\lvert a\rvert}$ for all $z\in U$ as $n\to\infty$. Hence $f^n(z)\to \frac{a}{\lvert a\rvert}$ for all $z\in\mathbb{D}$ as $n\to\infty$ as $\mathbb{D}\subset U$.
\end{proof}

\begin{rem}\label{reme4}
    We can always choose a sequence of distinct elements $\{a_k\}$ such that $f_k(z)=\frac{z^k+a_k}{1+\overline{a_k}z^k}$ satisfies conditions of Theorem \ref{le2} \textit{(iii)}. For any $k\geq 2$, we choose $a_k$ such that $a_k=\frac{k-1}{k+1}\alpha_{k-1}$ such that $\alpha_{k-1}$ is $(k-1)^{th}$ root of unity which is not a $m$-th root of unity for any $m<(k-1)$. Then $\lvert a_k\rvert = \frac{k-1}{k+1}<1$. Now $\frac{a_k}{\lvert a_k\rvert}=\alpha_{k-1}$ which satisfies all the conditions. 
\end{rem}
\begin{exmp}
    For $k=3$, choose $a=\frac{1}{2}$, then $\frac{a}{\lvert a\rvert}=1$ which is a square root of unity. So, for $f(z)=\frac{z^3+\frac{1}{2}}{1+\frac{1}{2}z^3}=\frac{2z^3+1}{2+z^3}$, $f^n(z)\to 1$ for all $z\in\mathbb{D}$ as $n\to\infty$.
\end{exmp}
\begin{exmp}\label{example4}
    Let $F=\{f_2,f_3,....\}$ such that $f_k$, $k\geq 2$ is chosen as described in Remark \ref{reme4}. Define $G=<F>$. Then for every $k$, $f_{k}^n(z)\to \frac{a_k}{\lvert a_k\rvert}$ for all $z\in\mathbb{D}$ as $n\to \infty$. So, $\Bigl\{\frac{a_k}{\lvert a_k\rvert}: k\geq 2\Bigl\}\subset DW(G)$. As $\Bigl\{\frac{a_k}{\lvert a_k\rvert}: k\geq 2\Bigl\}$ is an infinite set. So $DW(G)$ is infinite. 
\end{exmp}
We now prove a theorem related to the cardinality of the set $DW(G)$ for a finitely generated Abelian rational semigroup $G$.
\begin{thm}\label{theorem1}
Let $G$ be a finitely generated Abelian rational semigroup. Then $DW(G)$ is countable.
\end{thm}
\begin{proof}
First, we show that $G$ is countable. Let $G=<f_1,f_2,\dots,f_n>$. Now any $f\in G$ can be written as $f=f{_1}^{m_1}\circ f{_2}^{m_2}\circ\dots\circ f{_n}^{m_n}$ where $m_i\in\mathbb{N}\cup\{0\}$ for all $i=1,2,\dots,n$ as $G$ is Abelian. Define $\tau:G\rightarrow\{\mathbb{N}\cup\{0\}\}^n$ such that $\tau(f)=(m_1,m_2,\dots,m_n)$ where $f=f{_1}^{m_1}\circ f{_2}^{m_2}\circ\dots\circ f{_n}^{m_n}$. We now show that the function is well-defined. Let $f=g$. So $f=f{_1}^{m_1}\circ f{_2}^{m_2}\circ\dots\circ f{_n}^{m_n}=g$. Thus, $\tau(f)=\tau(g)=(m_1,m_2,\dots,m_n)$. So the function is well-defined.
Now let us assume that $f=f{_1}^{m_1}\circ f{_2}^{m_2}\circ \dots\circ f{_n}^{m_n}$, $g=f{_1}^{k_1}\circ f{_2}^{k_2}\circ\dots\circ f{_n}^{k_n}$ and  $\tau(f)=\tau(g)$. Now $\tau(f)=\tau(g)$ implies that $(m_1,m_2,\dots,m_n)=(k_1,k_2,\dots,k_n)$. So $m_i=k_i$ for all $i=1,2,\dots,n$. This implies $f=f{_1}^{m_1}\circ f{_2}^{m_2}\circ\dots\circ f{_n}^{m_n}=f{_1}^{k_1}\circ f{_2}^{k_2}\circ\dots\circ f{_n}^{k_n}=g$. So $\tau$ is one-one function. Let $(m_1,m_2,\dots,m_n)\in \{\mathbb{N}\cup\{0\}\}^n$. Now we have $f=f{_1}^{m_1}\circ f{_2}^{m_2}\circ\dots\circ f{_n}^{m_n}\in G$ such that $\tau(f)=(m_1,m_2,\dots,m_n)$. Thus, $\tau$ is also surjective and hence a bijective function. So $card(G)=card(\{\mathbb{N}\cup\{0\}\}^n)$. As finite product of countable sets is countable and $\mathbb{N}\cup\{0\}$ is countable, so $\{\mathbb{N}\cup\{0\}\}^n$ is countable. So $G$ is countable.
Now we show that $DW(G)$ cannot be uncountable for finitely generated Abelian rational semigroups.
If for all $f\in G$, $f^n(\mathbb{D})$ does not converge to a point in $\overline{\mathbb{D}}$ then $DW(G)$ is empty by definition and hence has cardinality $0$. Let us assume that $G$ has a non-empty subset $S$ which contains all $f\in G$ such that $f^n(\mathbb{D})$ converges to a point in $\overline{\mathbb{D}}$. Now for any $f\in S$ let us assume that $z_f$ is the point in $\overline{\mathbb{D}}$ to which $f^n(z)$ converges for all $z\in\mathbb{D}$. Define $\zeta:S\to DW(G)$ by $\zeta(f)=z_f$. The function is well-defined by definition. Let $z^{\prime}\in DW(G)$. Then by definition of Denjoy-Wolff like set there exists a function $f^{\prime}\in G$ such that $f^{{\prime}^n}(z)\rightarrow z^{\prime}$ as $n\rightarrow\infty$ for all $z\in\mathbb{D}$. Hence, by definition of $S$, we have $f^{\prime}\in S$. So, for any point in the co-domain of $\zeta$ there exists a $\zeta-$preimage of the point. So, $\zeta$ is surjective. Hence $card(DW(G))\leq card(S) \leq card(G)$. We know that $G$ is countable. So, Denjoy-Wolff like set of $G$ can be at most countable.\\
\end{proof}
We propose the following conjecture which if proven could strengthen the above result.
\begin{conj}
Let $G$ be a finitely generated Abelian rational semigroup. Then $DW(G)$ is empty or singleton.
\end{conj}
This conjecture is false for Abelian semigroups of rational maps under composition in general. This can be seen in the following example.
\begin{exmp}
Let $G=<f,g>$ where $f(z)=1$ and $g(z)=z^2$. We see that $f\circ g=g\circ f= 1$. So $G$ is a finitely generated Abelian semigroup of rational maps. Now $f^n(z)\rightarrow 1$ and $g^n(z)\rightarrow 0$ as $n\rightarrow \infty$ for all $z\in\mathbb{D}$. So $DW(G)=\{1,0\}$. But $G$ is not a rational semigroup as degree of $f$ is less than $2.$ 
\end{exmp}

\section{Properties of Denjoy-Wolff like set}\label{s3}
We now discuss some properties of the Denjoy-Wolff like set for various rational semigroups. First, we show that when $DW(G)$ lies inside the unit disk then $DW(G)$ cannot be finite with cardinality $>1$ for Abelian rational semigroup $G$.
\begin{thm}
    Let $G$ be an Abelian rational semigroup. If $DW(G)\subset \mathbb{D}$ and $DW(G)\neq\emptyset$ then $DW(G)$ is either infinite or singleton.
\end{thm}
\begin{proof}
    If $DW(G)$ is infinite, we have nothing to prove. Let $DW(G)$ be finite. Assume $DW(G)=\{z_1,z_2,....z_n\}$. So there exists $\{f_1,f_2,...f_n\}\subset G$ such that $f_i^n(z)\to z_i$ as $n\to\infty$ for all $z\in\mathbb{D}$ and $i=1,2,...,n$. Let $h=f_1\circ f_j\in G$ for some $j\in\{1,2,...,n\}$. Now $h^k(z)=(f_1\circ f_j)^k(z)=(f_1^k\circ f_j^k)(z)=(f_j^k\circ f_1^k)(z)$ as $G$ is an Abelian semigroup. Now $(f_1^k\circ f_j^k)(z)\to z_1$ and $(f_j^k\circ f_1^k)(z)\to z_j$ as $k\to\infty$ for all $z\in\mathbb{D}$. This implies $h^k(z)\to z_1=z_j$ as $k\to\infty$ for all $z\in\mathbb{D}$. As we have chosen $j$ arbitrarily, $z_j=z_1$ for all $j=2,3,...,n$. So, $DW(G)=\{z_1\}$ and is singleton.
\end{proof}
The next theorem shows that when the open unit disk intersects the Julia set then $DW(G)$ is empty for any finitely generated Abelian rational semigroup.
\begin{thm}
Let $G$ be a finitely generated Abelian rational semigroup. If $\mathbb{D}\cap J(G)\neq \emptyset$ then $DW(G)$ is empty.
\end{thm}
\begin{proof}
As $G$ is a finitely generated Abelian rational semigroup so for any $f\in G$, $J(G)=J(f)$ \cite{hinkkanen1996dynamics}. Now let $z_0\in DW(G)$. This implies that there is a $g\in G$ such that $g^n(z)\rightarrow z_0$ for all $z\in\mathbb{D}$ as $n\rightarrow \infty$. Then $\mathbb{D}\subset F(g)$ for some $g\in G$. This is a contradiction since $\mathbb{D}\cap J(G)=\mathbb{D}\cap J(g)\neq\emptyset$. So $DW(G)$ must be empty.
\end{proof}
We can remove the restriction of $G$ being finitely generated Abelian by introducing a new condition.
\begin{thm}
    Let $G$ be a rational semigroup such that $\mathbb{D}\subset J(g)$ for all $g\in G$. Then $DW(G)$ is empty.
\end{thm}
\begin{proof}
    Let $z_g\in\mathbb{D}$. As $\mathbb{D}\subset J(g)$, the family $\{g^n\}$ is not normal at any neighbourhood of $z_g$. So, the sequence $\{g^n(z)\}$ can not converge for a point $z$ in any neighbourhood $U$ of $z_g$. Let us choose a neighbourhood $U_g$ of $z_g$ such that $U_g\subset \mathbb{D}$. So $\{g^n(z)\}$ does not converge for at least one point $z$ in $U_g$. So for any $f\in G$, there exists at least one point $z$ in $\mathbb{D}$ for which $\{f^n(z)\}$ does not converge. So $DW(G)$ must be empty.
\end{proof}
The next result concerns the Denjoy-Wolff like set for certain polynomial semigroups. Class $\mathcal {G}$ is the class of all polynomial semigroups which are postcritically bounded and all elements of the semigroup are of degree greater than 1 \cite{sumi2007dynamics}.
\begin{thm}
    Let $G\in\mathcal{G}$ be a polynomial semigroup. Suppose $z_f\in\mathbb{D}$ be the Denjoy-Wolff point of some $f\in G$. If $f^{\prime}(z_f)=0$ then either $z_f\in F(g)$ for all $g\in G$ or $DW(G)\cap J(G)^\circ\neq\emptyset$.
\end{thm}
\begin{proof}
    Since $f^n(z)\to z_f\in\mathbb{D}$ for all $z\in\mathbb{D}$ as $n\to\infty$, $f$ has a fixed point at $z_f$. Also as $f^{\prime}(z_f)=0$, $z_f$ is super-attracting fixed point of $f$. Now either $z_f\in F(g)$ for all $g\in G$ or $z_f\in J(g)$ for at least one $g\in G$. In the first case we have nothing to prove. Let $z_f\in J(g)$ for some $g\in G$. As $J(g)\subset J(G)$, $z_f\in J(G)$. So by [\cite{sumi2007dynamics},Theorem 2.21], we can say that $z_f\in J(G)^\circ$. Since by definition $z_f\in DW(G)$, we have $z_f\in DW(G)\cap J(G)^\circ$.
\end{proof}

Now we discuss the cardinality of $DW(G)$ for a type of semigroup $G$ containing functions of a certain category. By $\mathbb{D}(0,r)$ we mean the open disk of radius $r$ with center at the origin. Let $F_r=\{f:\mathbb{C}\rightarrow\mathbb{C}: f:\mathbb{D}(0,r)\rightarrow \mathbb{D}(0,r)$ be analytic and $f(0)=0,\text{ }0<r<1\}$.
\begin{thm}\label{th34}
Let $G=<f_1,f_2,\dots>$ be a rational semigroup under composition such that $f_i\in F_r$ for all $i=1,2,\dots$ and $r$ is fixed real number such that $0<r<1.$ Let $B$ be a set with non-empty interior which is a union of circles centered at the origin and which contains a sequence of circles tending to the origin. If $B$ is forward invariant in $G$, then $DW(G)$ is either empty or $\{0\}$.
\end{thm}
Before we discuss the proof of the theorem, we first show that a semigroup $G$ and a set $B$ exist which satisfies the above conditions. Let $f,\text{ }g:\mathbb{C}\rightarrow\mathbb{C}$ be two polynomial functions defined as $f(z)=\frac{z^2}{r}$ and $g(z)=\frac{z^3}{r^2}$ where $0<r<1$. Now when $z\in\mathbb{D}(0,r)$ then we see that $\lvert f(z)\rvert=\frac{\lvert z\rvert^2}{r}<r$ and $\lvert g(z)\rvert=\frac{\lvert z\rvert^3}{r^2}<r$. So $f(\mathbb{D}(0,r))\subset \mathbb{D}(0,r)$ and $g(\mathbb{D}(0,r))\subset \mathbb{D}(0,r)$. Also $f(0)=g(0)=0$ and both functions are analytic in $\mathbb{D}(0,r)$. So, $f,\text{ }g\in F_r$. Now let $G=<f,g>$. Then we see that $(f\circ g)(z)=f(g(z))=f\Bigl(\frac{z^3}{r^2}\Bigl)=\frac{z^6}{r^5}$. Also, $(g\circ f)(z)=g(f(z))=g\Bigl(\frac{z^2}{r}\Bigl)=\frac{z^6}{r^5}$. So $G$ is an Abelian semigroup, and we can write any element of $G$ as $f^k\circ g^s$ for some non-negative integer $k,\text{ }s$. Now $(f^k\circ g^s)(z)=f^k\biggl(\frac{z^{3^s}}{r^{2\frac{3^s-1}{3-1}}}\biggl)=\frac{\Bigl(\frac{z^{3^s}}{r^{(3^s-1)}}\Bigl)^{2^k}}{r^{2^k-1}}=\frac{z^{3^s.2^k}}{r^{(3^s.2^k-1)}}$. Let $M=\{n\in\mathbb{N}:$ either $3\lvert n$ or $2\lvert n$\}. So, every $h\in G$ is of the form $h(z)=\frac{z^l}{r^{l-1}}$ for some $l\in M$. Now we choose the set $B$. Let $B=\cup_{j=1}^{\infty}\{z:\lvert z\rvert=r^j$ where $j\in\mathbb{N}\}$. Clearly $B$ has empty interior and $B$ is a union of circles centered at origin. Also $C_j=\{z:\lvert z\rvert=r^j\}$ for $j=1,2,\dots$ is a sequence of circles lying in $B$ which tends to $0$ as $j\rightarrow\infty$ as $0<r<1$. Now we show that $B$ is forward invariant for any $h\in G$. Let $z\in B$. So there exists a natural number $k$ such that $\lvert z\rvert=r^k$. Let $h\in G$. We have already shown that $h(z)=\frac{z^l}{r^{l-1}}$ for some $l\in M$. So $\lvert h(z)\rvert =\Bigl\lvert \frac{z^l}{r^{l-1}}\Bigl\rvert=\frac{r^{kl}}{r^{l-1}}=r^{kl-l+1}$. So, $h(z)\in C_{kl-l+1}\subset B$ as $kl-l+1\in\mathbb{N}$. As $z$ is chosen arbitrarily so $h(B)\subset B$. So, $B$ is forward invariant under $h$. Now as $h$ is arbitrary, so $B$ is forward invariant in $G$. \\
Now we prove Theorem \ref{th34}. To prove it we require the following Lemma 6 of \cite{stankewitz1999completely}.
\begin{lemm}\label{lem32}
Let $L:\mathbb{D}(0,r)\rightarrow \mathbb{D}(0,r)$ where $0<r<1$, be an analytic function such that $L(0)=0$. Let $B$ be a set with empty interior which is a union of circles centered at origin and which contains a sequence of circles tending to $0$. If $B$ is forward invariant under the map $L$, then $L$ is of the form $L(z)=az^j$ for some $a\neq 0$ and non-negative integer $j$.
\end{lemm}
\begin{proof}[Proof of Theorem \ref{th34}]
We have that $G=<f_1,f_2,\dots>$ with $f_i\in F_r$ for $i=1,2,\dots$ and fixed real number $r$ with $0<r<1$. Now if $B$ is forward invariant in $G$ then $B$ is forward invariant for every $g\in G$. So, by previous lemma, $g(z)=az^j$ for some $a\neq 0$ and non-negative integer $j$. Now $g^n(z)=a^{\frac{j^n-1}{j-1}}z^{j^n}$. So when $z\in\mathbb{D}(0,r)$ we have $g^n(z)\rightarrow 0$. Now we have the following cases.\\ \\
\textbf{Case 1:} $g^n(z)\nrightarrow 0$ for all $z\in \mathbb{D}\setminus \mathbb{D}(0,r)$, for any $g\in G$.  So, in this case for any $g\in G$, $g^n(z)$ does not have a unique limit for all $z\in\mathbb{D}$. So $DW(G)$ is empty.\\ \\
\textbf{Case 2:} $g^n(z)\rightarrow 0$ for all $z\in \mathbb{D}\setminus \mathbb{D}(0,r)$ for some $g\in G$. So there exists a $g\in G$ such that $g^n(z)\rightarrow 0$ for all $z\in\mathbb{D}$ as $n\rightarrow\infty$. So $DW(G)=\{0\}$.
\end{proof}

\section{The class $\Psi$}\label{section4}
We now introduce a special type of rational semigroups. Let $G$ be a rational semigroup such that for any $f\in G$, $f(\mathbb{D})\subset \mathbb{D}$ and $f^n(z)\to z_f$ for some $z_f\in \overline{\mathbb{D}}$ for all $z\in\mathbb{D}$ as $n\to\infty$. Let the class of all rational semigroups which satisfy the above condition be denoted by $\Psi$. The class $\Psi$ is not empty as Example \ref{example1} belongs to this class. Also, it is clear by definition that for any rational semigroup $G\in\Psi$, $DW(G)\neq\emptyset$. In the next result, we give a theorem to partition any rational semigroup $G\in\Psi$.
\begin{thm}\label{theorem35}
    Let $G\in\Psi$. Then $G$ can be partitioned into $\lvert DW(G)\rvert$ partitions where $\lvert A\rvert$ is the cardinality of $A$.
\end{thm}
\begin{proof}
    Since $G\in\Psi$, for any $f\in G$ there exists a $z_f\in\overline{\mathbb{D}}$ such that $f^n(z)\to z_f$ as $n\to\infty$ for all $z\in\mathbb{D}$. Then $DW(G)=\{z_f\in\overline{\mathbb{D}}: f\in G\}$. Now define a relation $\sim$ on $G$ as $g\sim f$ if and only if $z_f=z_g$. As $z_f=z_f$, $f\sim f$ and $\sim$ is reflexive. As $z_f=z_g\implies z_g=z_f$, so $g\sim f\implies f\sim g$ and $\sim$ is symmetric. Also $f\sim g$ and $g\sim h$ implies $z_f=z_g=z_h$ which imply $\sim$ is transitive. So, $\sim$ is an equivalence relation on $G$. Thus it partitions $G$ into some equivalence classes. Let $[f]\subset G$ be the equivalence class $\{g\in G:g\sim f\}$. Let $G/\sim$ = $S=\{[f]:f\in G\}$ which is the collection of all the equivalence classes of $G$. Define $r:S\to DW(G)$ as $r([f])=z_f$. Now let $[f]=[g]$ which implies $g\sim f$ and so $z_f=z_g$ and $r([f])=r([g])$. So, $r$ is well-defined. Now let $r([f])=r([g])$ which implies $z_f=z_g$ for $f,g\in G$. Then $g\sim f$ and $[g]=[f]$. So, $r$ is one-one. Also, by definition $r$ is surjective and so is bijective. Thus $\rvert DW(G)\lvert=\rvert S\lvert$. So, $G$ can be partitioned into $\lvert DW(G)\rvert$ partitions.
\end{proof}
\begin{cor}
Let $G\in \Psi$ be finitely generated Abelian. Then if \\
a) $DW(G)$ is finite then there exists at least one $[f]\in G/\sim$ such that $[f]$ is countably infinite.\\
b) every $[f]\in G/\sim$ is finite then $DW(G)$ is infinite.
\end{cor}
\begin{proof}
    Since $G$ is finitely generated Abelian so from the proof of Theorem \ref{theorem1}, $G$ is countably infinite. Also as $G\in \Psi$, from Theorem \ref{theorem35} we say that $\lvert G\rvert =\Sigma \lvert [f]\rvert$ where summation is taken over the $\lvert DW(G)\rvert$ partitions of $G$. Now the proof follows since finite sum of finite numbers is also finite, and $G$ is countably infinite. 
\end{proof}
Since when $G\in\Psi$ then $DW(G)$ is always non-empty, we  use $DW(G)$ to classify the class $\Psi$. When $G\in \Psi$ then 3 cases can occur. We call $G$,
\begin{itemize}
    \item Absorbing if $DW(G)\subset \mathbb{D}$.
    \item Dispersing if $DW(G)\subset\partial\mathbb{D}$.
    \item Hybrid if $DW(G)\cap \mathbb{D}\neq\emptyset$ and $DW(G)\cap\partial \mathbb{D}\neq\emptyset$.
\end{itemize}
\begin{exmp}
    Let $G=<f,g>$ where $f(z)=z^2$ and $g(z)=z^3$. Clearly $G\in\Psi$. Then $DW(G)=\{0\}$. So $G$ is absorbing.
\end{exmp}
\begin{thm}\label{th42}
    Let $G\in\Psi$ and let $G_\phi=\{\phi\circ g\circ \phi^{-1}:g\in G\}$ where $\phi$ is a continuous function from $\widehat{\mathbb{C}}\to\widehat{\mathbb{C}}$ and $\phi\lvert_{\mathbb{D}}$ is bijection from $\mathbb{D}$ to $\mathbb{D}$. Then $G$ is absorbing if and only if $G_\phi$ is absorbing.
\end{thm}
\begin{proof}
    We have $G\in\Psi$. First, we show that $G_\phi\in\Psi$ also. Let $h\in G_\phi$. So there exists a $g\in G$ such that $h=\phi\circ g\circ\phi^{-1}$. So $h^n(z)=\phi\circ g^n\circ\phi^{-1}$ on $\mathbb{D}$. Now as $g\in G$, $g^n(z)\to z_g\in\overline{\mathbb{D}}$ as $n\to\infty$ for all $z\in\mathbb{D}$ for some $z_g$. So $h^n(z)\to\phi(z_g)$ as $n\to\infty$ for all $z\in\mathbb{D}$. We have to show that $\phi(z_g)\in\overline{\mathbb{D}}$. Let us assume $\phi(z_g)\in\widehat{\mathbb{C}}\setminus \overline{\mathbb{D}}$. Since $\widehat{\mathbb{C}}\setminus \overline{\mathbb{D}}$ is an open set, $\phi(z_g)$ has a neighbourhood $U$ such that $U\subset \widehat{\mathbb{C}}\setminus \overline{\mathbb{D}}$. Now let $z^\prime_n=g^n(z^\prime)$ for some $z^\prime\in \mathbb{D}$ and $n\in\mathbb{N}$. So, $\{z^\prime_n\}\subset \mathbb{D}$ converges to $z_g$. By continuity of $\phi$, $\phi(z^\prime_n)\to \phi(z_g)$. So any neighbourhood of $\phi(z_g)$ must contain elements of $\{\phi(z^\prime_n)\}$. But $U$ cannot contain any element of $\{\phi(z^\prime_n)\}$ as $\phi(z^\prime_n)\in\mathbb{D}$ for all $n\in\mathbb{N}$. This is a contradiction. So, $\phi(z_g)\in\overline{\mathbb{D}}$. Hence $G_\phi\in\Psi$ as $h$ is chosen arbitrarily. Now let $G$ be absorbing. So $DW(G)\subset\mathbb{D}$. From the proof so far, it is clear that $DW(G_\phi)=\phi(DW(G))$. Since $\phi\lvert_\mathbb{D}$ is a bijection in $\mathbb{D}\to\mathbb{D}$ and $DW(G)\subset\mathbb{D}$, $\phi(DW(G))\subset \mathbb{D}$. So $DW(G_\phi)\subset\mathbb{D}$ and $G_\phi$ is absorbing. The converse part follows in a similar approach considering $DW(G)=\phi^{-1}(DW(G_\phi))$.
\end{proof}
\begin{cor}
    Let $G\in\Psi$ and $G_{\phi}$ be defined as in Theorem \ref{th42}. Then $G$ is dispersing if and only if $G_\phi$ is dispersing.
\end{cor}
\begin{proof}
    Let $G$ be dispersing. Due to Theorem \ref{th42}, $G_\phi$ is not absorbing. So either $G_\phi$ is dispersing or hybrid. If $G_\phi$ is hybrid, then there exists at least one $h\in G_\phi$ such that $h^n(z)\to z_h\in\mathbb{D}$. So from the proof of Theorem \ref{th42}, it is clear that there exists a $g\in G$ such that $h=\phi\circ g\circ\phi^{-1}$ and $g^n(z)\to \phi^{-1}(z_h)$ as $n\to \infty$ and $z\in\mathbb{D}$. But as $\phi\lvert_\mathbb{D}$ is a bijection in $\mathbb{D}\to\mathbb{D}$, $\phi^{-1}(z_h)\in \mathbb{D}$. This contradicts that $G$ is dispersing. So $G_\phi$ can not be hybrid and hence must be dispersing. The converse part follows a similar approach.
\end{proof}
 \section*{Conclusion}
 In this paper, we discuss various properties of the Denjoy-Wolff like set for a rational semigroup $G$. In future, we wish to classify rational semigroups with the help of this set. The classification of semigroups of class $\Psi$ can be investigated on a more detailed account which can offer better insight into the dynamics of some special types of rational maps. 

\subsection*{Acknowledgment}
The second author sincerely acknowledges the financial support rendered by the National Board for Higher Mathematics, Department of Atomic Energy, Government of India sponsored project with Grant No. \\02011/17/2022/NBHM(R.P)/R\&D II/9661 dated: 22.07.2022. The third author acknowledges the financial support received from
National Board for Higher Mathematics, Government of India, Post Doctoral
Fellowship (No. 0204/61/2017/R\&D II/15334).
\subsection*{Statements and Declarations}
 No potential competing interest was reported by the authors.
 \subsection*{Data availability statement}
 Data sharing not applicable to this article as no datasets were generated or analysed during the current study.
%\begin{thebibliography}{1}
%\bibitem{test} A. B. C. Test, \textit{On a Test.} J. of Testing
%\textbf{88} (2000), 100--120.
%\bibitem{latex} G. Gr\"atzer, \textit{Math into \LaTeX.} 3rd Edition,
%Birkh\"auser, 2000.
%\end{thebibliography}
\bibliographystyle{acm}
\bibliography{bjourdoc}

% ------------------------------------------------------------------------
\end{document}